\documentclass{article}

     \PassOptionsToPackage{numbers, compress}{natbib}
     \usepackage[preprint]{neurips_2021}

\usepackage[utf8]{inputenc} % allow utf-8 input
\usepackage[T1]{fontenc}    % use 8-bit T1 fonts
\usepackage{hyperref}       % hyperlinks
\usepackage{url}            % simple URL typesetting
\usepackage{booktabs}       % professional-quality tables
\usepackage{amsfonts}       % blackboard math symbols
\usepackage{nicefrac}       % compact symbols for 1/2, etc.
\usepackage{microtype}      % microtypography
\usepackage{xcolor}         % colors

\usepackage{graphicx}
\usepackage{subfigure}
\usepackage{color}
\usepackage{amsmath,amsxtra,amsfonts,amscd,amssymb,bm,amsthm}
\newtheorem{theorem}{Theorem}

\newtheorem{assumption}[theorem]{Assumption}
\newtheorem{lemma}[theorem]{Lemma}

\newcommand{\be}{\begin{equation}}
\newcommand{\ee}{\end{equation}}
\newcommand{\bee}{\begin{equation*}}
\newcommand{\eee}{\end{equation*}}

\newcommand{\beaa}{\begin{eqnarray*}}
\newcommand{\eeaa}{\end{eqnarray*}}

\newcommand{\Rn}{\mathbb{R}}

\usepackage{enumerate}
\usepackage[algo2e,linesnumbered,vlined,ruled]{algorithm2e}

\newcommand{\diag}{\operatorname{diag}}

\newcommand{\W}{\Theta}
\newcommand{\G}{\mathcal{G}}
\newcommand{\FOM}{\mathbf{G}}

\title{NG+ : A Multi-Step Matrix-Product Natural Gradient Method for Deep Learning}

\author{%
Minghan Yang \\ Beijing International Center for Mathematical Research\\ Peking Univerisity \\ \texttt{yangminghan@pku.edu.cn}
\And
Dong Xu\\  Beijing International Center for Mathematical Research\\ Peking Univerisity \\ \texttt{taroxd@pku.edu.cn}
\And
Qiwen Cui\\  School of Mathematical Sciences \\ Peking Univerisity \\ \texttt{qwcui1107@gmail.com}
\And
Zaiwen Wen\\  Beijing International Center for Mathematical Research\\ Peking Univerisity \\ \texttt{wenzw@pku.edu.cn}
\And
Pengxiang Xu\\  Peng Cheng Laboratory \\ \texttt{xupx@pcl.ac.cn}
}

\begin{document}
\maketitle

\begin{abstract}
In this paper, a novel second-order method called NG+ is proposed. By following the rule ``the shape of the gradient equals the shape of the parameter", we define a generalized fisher information matrix (GFIM) using the products of gradients in the matrix form rather than the traditional vectorization. Then, our generalized natural gradient direction is simply the inverse of the GFIM multiplies the gradient in the matrix form. Moreover, the GFIM and its inverse  keeps the same for multiple steps so that the computational cost can be controlled and is comparable with the first-order methods. A global convergence is established under some mild conditions and a regret bound is also given for the online learning setting. Numerical results on image classification with ResNet50, quantum chemistry modeling with Schnet,  neural machine translation with Transformer and recommendation system with DLRM illustrate that GN+ is competitive with the state-of-the-art methods.
\end{abstract}

\section{Introduction}
%\begin{itemize}
%\item Briefly Review first-order method and subsampled Newton/ Newton Sketch and stochastic quasi-Newton methods. 
%\item Review KBFGS and other second-order methods. 
%\item Compare with SENG/Shampoo/KFAC in detail. Operations./ Synchronous/ Inverse.
%\end{itemize}

Optimization methods play an important role in deep learning problems. %They can be briefly split into two classes: first-order and second-order type methods.
The first-order stochastic optimization methods, e.g., SGD \cite{RobMon51}, AdaGrad
\cite{duchi2011adaptive} and Adam \cite{kingma2014adam}, have been broadly used in
practical applications and well-optimized in deep learning frameworks such as
PyTorch \cite{paszke2019pytorch} and TensorFlow \cite{abadi2016tensorflow}. It is a challenge to push forward  their performance in terms of iterations and computational time. The generalization gap in large-batch training \cite{KeskarMNST17,shallue2019measuring} has been partly addressed in LARS \cite{you2017LARS}.

%Traditional second-order type methods, such as Newton method and quasi-Newton method \cite{nocedal2006numerical}, have shown  advantages in quite a few applications and usually enjoy fast local convergence rates. The key points are constructing a good approximation to the Hessian matrix and efficient evaluation of its inverse.
Recently, quite a few different stochastic second-order methods have been
developed for large-scale problems. The efficiency of  the subsampled Newton method
\cite{roosta2019sub},  Newton sketch method \cite{pilanci2017newton}, stochastic
quasi-Newton method \cite{byrd2016stochastic}, structured stochastic
quasi-Newton method  \cite{yang2020S2QN} %, online natural gradient method  \cite{roux2008topmoumoute} and
 and Kronecker-factored quasi-Newton method \cite{goldfarb2020practical,ren2021kronecker} remains to be verified in large-scale deep learning tasks.
%However, they do not consider the specific structures of deep learning. The Hessian-free method \cite{HF-DL} use fast Hessian matrix-vector productions and conjugate-gradient (CG) method to generate a second-order type direction. The authors in \cite{practicalGN} use a block-diagonal matrix to approximate to the Gauss-Newton matrix for the fully-connected network. TONGA\cite{roux2008topmoumoute} maintains a low-rank approximation to the fisher information matrix. By approximating the inverse of the Hessian matrix at each layer, the quasi-Newton methods are used in \cite{yang2020S2QN, ren2021kronecker, goldfarb2020practical}.
%There are three practical second-order type methods, 
 On the other hand, KFAC \cite{martens2015optimizing}, SENG
 \cite{yang2020sketchy} and Shampoo \cite{anil2020second} have been successful
 in deep learning models whose scale is at least the same as ResNet50 on
 ImageNet-1k. Under the independency assumptions, KFAC approximates the Fisher
 information matrix (FIM) by a Kronecker product of two smaller matrices. The SENG
 method utilizes the low-rank property of the empirical Fisher information matrix
 and the structures of the gradients to construct a search direction in a small subspace
 by using sketching techniques. Shampoo uses the 1/4-th inverse of two (or the
 1/2 inverse of one) online “structured” matrices to precondition the flattened gradients.
%While the best result of first-order methods on ResNet50 with ImageNet-1k is achieved by using 64 epochs  to obtain 75.9\% top-1 validation accuracy \cite{nado2021large},
 It has been shown that second-order methods attain the
same validation accuracy as first-order methods with fewer epochs. For example, Shampoo
\cite{anil2020second} uses 44 epochs and SENG \cite{yang2020sketchy} takes only
41 epochs to achieve the same validation accuracy with a favorable computational
time.

In this paper, we develop a matrix-product natural gradient method NG+ for deep
learning problems. We view the parameters as a set of matrices and define a generalized Fisher information matrix (GFIM) in terms of the products of gradients in matrix
form. Consequently, a corresponding natural gradient direction is formulated.
Since the size of the GFIM is much smaller than that of the FIM in the vector
space, the inversion of GFIM is affordable and the main numerical algebraic
operations are greatly simplified compared with Shampoo. Although NG+ seems to
be a trivial modification of Shampoo, their concepts are significantly different
since Shampoo is a variant of full-matrix AdaGrad but NG+ is treated as an extension of the natural gradient method.  By using techniques such as lazy update of GFIM, block-diagonal
approximation and sketchy techniques, the overall cost of NG+ can even be comparable with
the first-order methods.  Global convergence analysis is established under
some mild conditions and a regret bound is given for a variant of NG+.
Numerical experiments on important tasks such as image classification, quantum chemistry modeling, neural machine translation and recommend system illustrate the advantages of our NG+ over the state-of-the-art methods.

\section{NG+ : A Generalized Natural Gradient Method}
For a given dataset $\{x_i,y_i\}_{i=1}^N$, consider the empirical risk minimization problem. The parameters of the neural networks are usually a set of matrices or even higher order tensors. Assume that the parameters $\Theta$ are matrices for simplicity and generality. Then, the problem is:
\be
\label{finite-sum}
\min_{\W \in \Rn^{m\times n}}\Psi(\W) = \frac{1}{N}  \sum_{i=1}^N  \psi(x_i,y_i,\W)  = \frac{1}{N}  \sum_{i=1}^N\psi_i(\W) ,\ee
where $\psi(x, y, \W)$ is the loss function. In deep learning problems, $\psi(\cdot)$ corresponds to the structures of neural networks. For example, if the output of data point $(x, y)$ through the network is $f(x;\W)$, then $\psi(x,y,\W) = \ell\left(f(x; \W),y\right)$ for some loss function $\ell(\cdot,\cdot)$, e.g., the mean squared loss and the cross-entropy loss. Denote the gradient for a single data sample by $\G_i = \nabla \psi_i(\W) \in \Rn^{m\times n} $ and the gradient of $k$-th iteration by $\G_{i,k} = \nabla \psi_i(\W_k)$. 

When treating the parameter $\Theta$ as a vector, the empirical Fisher Information Matrix (EFIM) is: 
\bee \text{EFIM}:= \frac{1}{N}\sum_{i=1}^N \text{vec}(\G_i)\text{vec}(\G_i)^\top,\eee where the vectorization of a matrix $A=(a_{i,j}) \in \Rn^{m\times n}$ is $\text{vec}(A) = [a_{1,1},\dots,a_{m,1},a_{1,2}, \dots,a_{m,2},\\ \dots,a_{1,n},\dots,a_{m,n}]^\top$. 
Since the gradient $\G_i$ itself is a matrix, by following the rule  ``the shape of the gradient equals the shape of the parameter", it is intuitive to define a generalized Fisher Information matrix (GFIM) as the average of the products between gradients in matrix form directly as follows: 
\begin{align}  
F = \frac{1}{N}\sum_{i=1}\G_i\G_i^\top, & \qquad \widetilde{F} =  \frac{1}{N}\sum_{i=1}\G_i^\top \G_i. \label{GFIM-2}
\end{align}

For a positive definite matrix $\mathbf{B}$, we have the steepest descent direction:
\be \label{generalized-2nd} - \frac{\mathbf{B}^{-1}\nabla \Psi(\Theta)}{\|\nabla \Psi\|_{{\mathbf{B}}^{-1}}} = \lim_{\epsilon \rightarrow 0}\frac{1}{\epsilon} \arg\min_{\|D\|_{\mathbf{B}}\leq \epsilon}\Psi(\Theta + D),\ee

where $\|A\|_{\mathbf{B}} = \sqrt{\text{tr}(A^\top \mathbf{B}A)}$.  Assume that $F\succ 0$, by letting $\mathbf{B} = F$ in \eqref{generalized-2nd}, we can obtain a generalized natural gradient direction as \bee D = - F^{-1}\nabla \Psi(\Theta).\eee
Similarly, by letting $\mathbf{B}=\widetilde{F}$ and adjusting the corresponding norm to $\|A\|_{\mathbf{B}} = \sqrt{\text{tr}(A \mathbf{B}A^\top)}$, the direction is changed to be $D = -\nabla \Psi(\Theta) \widetilde{F}^{-1}$. 
\subsection{Algorithmic Framework}
Now let us describe our multi-step framework. We compute one of the two subsampled curvature matrices every frequency $\mathcal{T}$ iterations on a sample set $S_k$ according to the size of the weights $\Theta$. Specifically, if $k$ \text{mod} $\mathcal{T}=0$, we calculate 
%\be
\begin{alignat} {2}
L_k & =  \frac{1}{|S_k|}\sum_{i\in S_k} \G_{i,k}\G_{i,k}^\top, \quad &\text{if } \ m  \leq n, \label{left-curvature-matrix} \\
R_k & = \frac{1}{|S_k|}\sum_{i\in S_k} \G_{i,k}^\top \G_{i,k}, \quad &\text{if } \  m> n. \label{right-curvature-matrix}
\end{alignat}
%\ee
Otherwise, we simply let either $L_k=L_{k-1}$ or $R_k = R_{k-1}$. Then, the direction is revised to be: 

  \be \label{direc-sec2}D_k=\left \{ \begin{array}{lr} -(\lambda_k I + L_k)^{-1} \FOM_k,&\text{if } \  m \leq n ,\vspace{1ex}\\  - \FOM_k (\lambda_k I + R_k)^{-1},&\text{if } \  m > n  ,\end{array} \right.\ee
  where $\lambda_k$ is a damping value to make $L_k$ or $R_k$ be positive definite and $\mathbf{G}_k $ is chosen to be the mini-batch gradient $\G_{B_k} = \frac{1}{|B_k|} \sum_{i\in B_k} \G_{i,k}$ given by the sample set $B_k$. 
%Note that when the variable is vector, the direction \eqref{direc-sec2} is equivalent to:
%  \be \label{direc-sec-vec}vec(D_k)=\left \{ \begin{array}{lr} -(I_{n\times n}\otimes L_k)^{-1} vec(\FOM_k),& m \leq n ,\vspace{1ex}\\ 
%   - ( R_k \otimes I_{m\times m})^{-1}vec(\FOM_k),&m \geq n .\end{array} \right.\ee
%
Finally, we update the parameter as\be \label{update-param}\W_{k+1} = \W_k+\alpha_k D_k .\ee
In summary, our proposed algorithm is shown in Algorithm \ref{alg:MatWM}.

 \LinesNumberedHidden
\begin{algorithm2e}[H]
%\small
\caption{NG+ : A Generalized Natural Gradient Method}
\label{alg:MatWM}
{\textbf{Inputs:}} Initial parameter $\W_0$, learning rates $\{\alpha_k\}_k$, regularization $\{\lambda_k\}_k,$ frequency $\mathcal{T}.$\\
\For{$k = 0,1,...,$}{
{Choose the sample set $B_k$ and compute $\FOM_k$;} \\
 \lIf{ $k$ mod $\mathcal{T} =0$}{ (\# compute the matrix)\\
   \quad {Construct $L_k$ or $R_k$ by \eqref{left-curvature-matrix} or \eqref{right-curvature-matrix} using the sample set ${S}_k$}}
 \lElse{ \\
   \quad Set $L_k = L_{k-1}$ or $R_k = R_{k-1}$}
{Compute the direction $D_k$ by \eqref{direc-sec2};}\\
{Update the parameter to $\Theta_{k+1}$ by \eqref{update-param}.}
}
\end{algorithm2e}

%Extension to the neural networks can be straightforward. Consider neural network with $L$ layers where the gradient with respect to (w.r.t.) the $l$-th layer is $\G_{i,k }^{l} \in \Rn^{m^l\times n^l}$ , $l \in [1,2,\dots, L]$. The direction for each layer can be computed as follows: 
%\be D_k^l=\left \{ \begin{array}{lr} -(L^l_k)^{-1} \FOM_k^l ,& m^l<n^l ,\vspace{1ex}\\  - \FOM^l_k (R^l_k)^{-1},&m^l \geq n^l ,\end{array} \right. \ee
%where $\FOM^l_k =\G_{B_k}^l := \frac{1}{|B_k|} \sum_{i\in B_k} \G^l_{i,k}$ is the mini-batch gradient w.r.t. the $l$-th layer.
%$L^l_k$ and $R^l_k$ are computed by replacing $\G_{i,k}^l$ with $\G_{i,k}$ in \eqref{left-curvature-matrix} and \eqref{right-curvature-matrix}. Note that throughout the paper, we use the superscript to indicate the layer number if no other statements.

\subsection{Interpretations of the EFIM and GFIM}
The gradient with respect to (w.r.t.) a single data point in a layer of a neural network often has the following special structures:
\be
\label{Grad-Struc}
\G_i =G_i(\Theta) A_i(\Theta)^\top,
\ee
where $G_i(\Theta)\in\Rn^{m\times \kappa}$ and $A_i(\Theta) \in \Rn^{n\times \kappa}.$
Note that we denote $G_i(\Theta_k)$ and $A_i(\Theta_k)$ by $G_{i,k}$ and $A_{i,k}$ if no confusion can arise.

Consider the fully-connected layer with $\kappa$ =1. By rewriting $G_i$
and $A_i$ as $g_i$ and $a_i$ for simplicity, we have $\G_i = g_ia_i^\top$ and use the following approximation:  
\be
\begin{aligned}
\text{vec}{(\G_i)} \text{vec}{(\G_i)}^\top &= (a_ia_i^\top ) \otimes (g_i g_i^\top)  \approx \left ( \|a_i\|_2^2  I \right) \otimes \left (g_i  g_i^\top\right)  =  I_{n\times n} \otimes \left (\G_i \G_i^\top\right).
\end{aligned}
\ee
For the whole dataset, we obtain: 
\be
\label{EFIM-GFIM-approx}
\begin{aligned}
\lambda_k I + \text{EFIM}%\\
%					&=\lambda_k I + \frac{1}{|B_k|}\sum_{i\in B_k}  (a_ia_i^\top ) \otimes (g_ig_i^\top) \\
%					 & \approx \lambda_k I + \frac{1}{|B_k|}\sum_{i\in B_k} (a_i^\top a_i)  I_{m\times m} \otimes \left (g_i g_i^\top\right) \\
%					  \approx& \lambda_k I + \frac{1}{|S_k|}\sum_{i\in S_k} I_{n\times n} \otimes \left (\G_i \G_i^\top\right)\\
					   \approx I_{n\times n } \otimes \left ( \lambda_k I + L_k \right )= I_{n\times n } \otimes \left (  \lambda_k I  + \frac{1}{N}\sum_{i=1}^N  \left (\G_i \G_i^\top\right)  \right ) ,
\end{aligned}
\ee
which illustrates that $L_k$ is an approximation to the EFIM in a certain sense. Similar results hold for $R_k$.
By using \eqref{EFIM-GFIM-approx}, we have 
\[
\left ( \lambda_k I + \text{EFIM}\right)^{-1}\text{vec}(\mathbf{G}_k) \approx \text{vec}\left ( (\lambda_k I + L_k)^{-1}\mathbf{G}_k\right),
\]
which means the direction in \eqref{direc-sec2} is a good approximation to the natural gradient direction in some sense.
When each column of $G_k$ are independent and identically distributed (i.i.d), the EFIM and $I\otimes \text{GFIM}$ are equivalent in expectation. We generate 2000 Gaussian random matrices with the size $200\times 200$ to compute EFIM and GFIM, respectively. The difference between the same diagonal block of the EFIM and $I\otimes \text{GFIM}$  is shown in Figure \ref{EFIM-GFIM-random}.
%The matrix $L_k$ is actually an upper bound of the subsampled FIM. From Lemma 9 in \cite{gupta2018shampoo}, we know that $vec{(\G_i)} vec{(\G_i)}^\top \preceq I_{n\times n} \otimes \left (\G_i \G_i^\top\right)$. Therefore, we have $\text{EFIM} \preceq I_{n\times n } \otimes L_k.$

\begin{figure}%[!t]
\centering
\begin{tabular}{ccc}
\includegraphics[width=4.5cm,height=3.375cm]{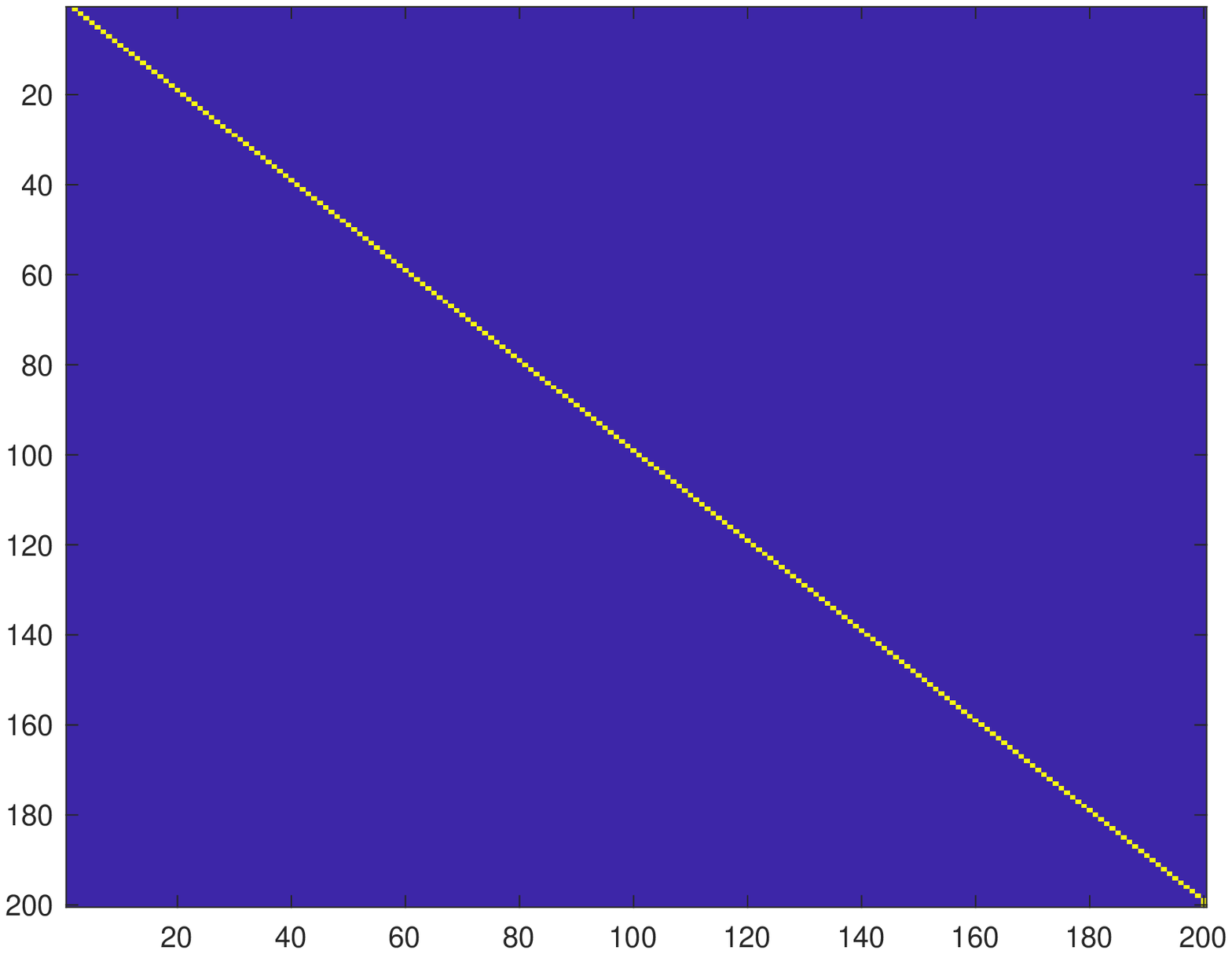}
\includegraphics[width=4.5cm,height=3.375cm]{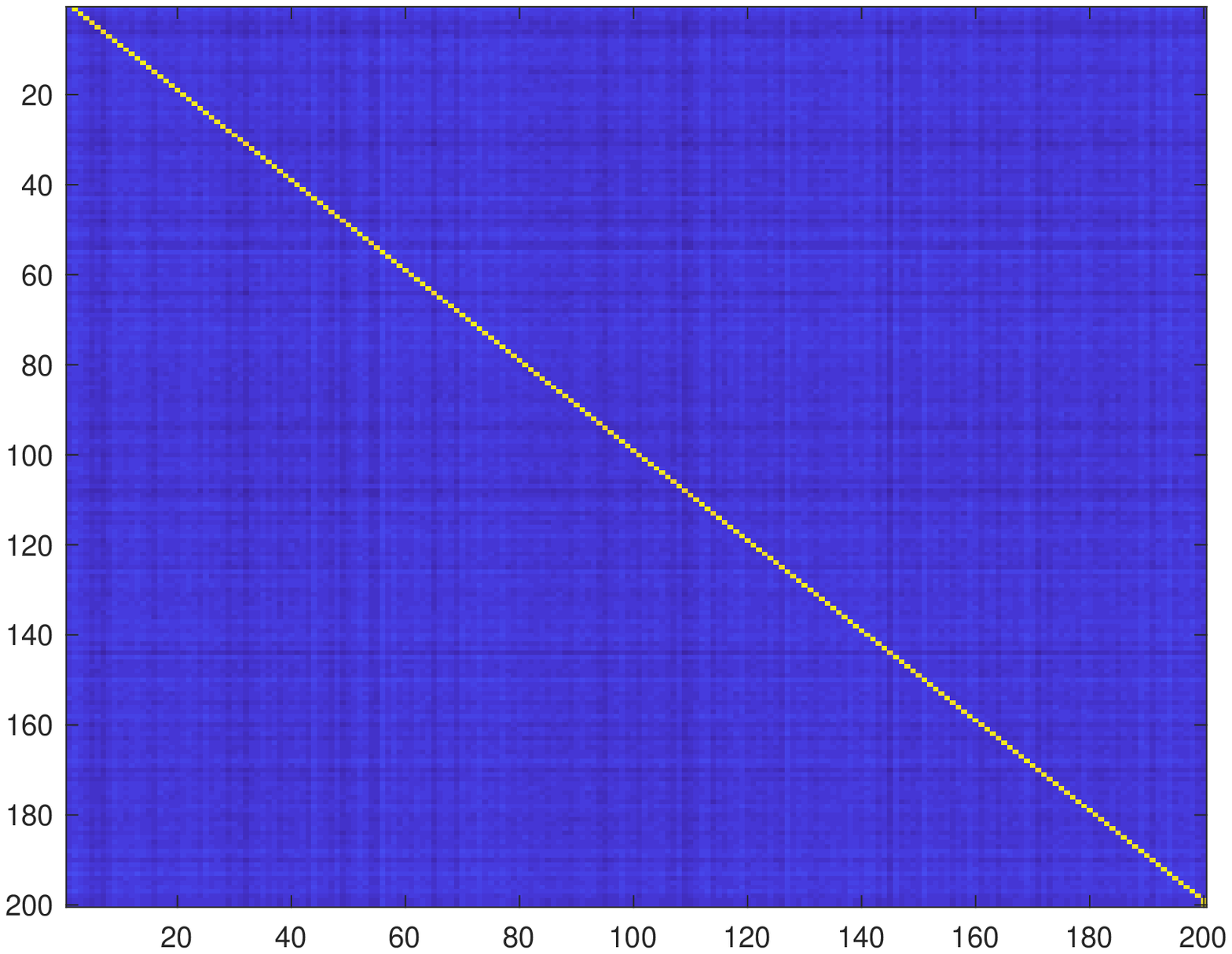}
\includegraphics[width=4.5cm,height=3.375cm]{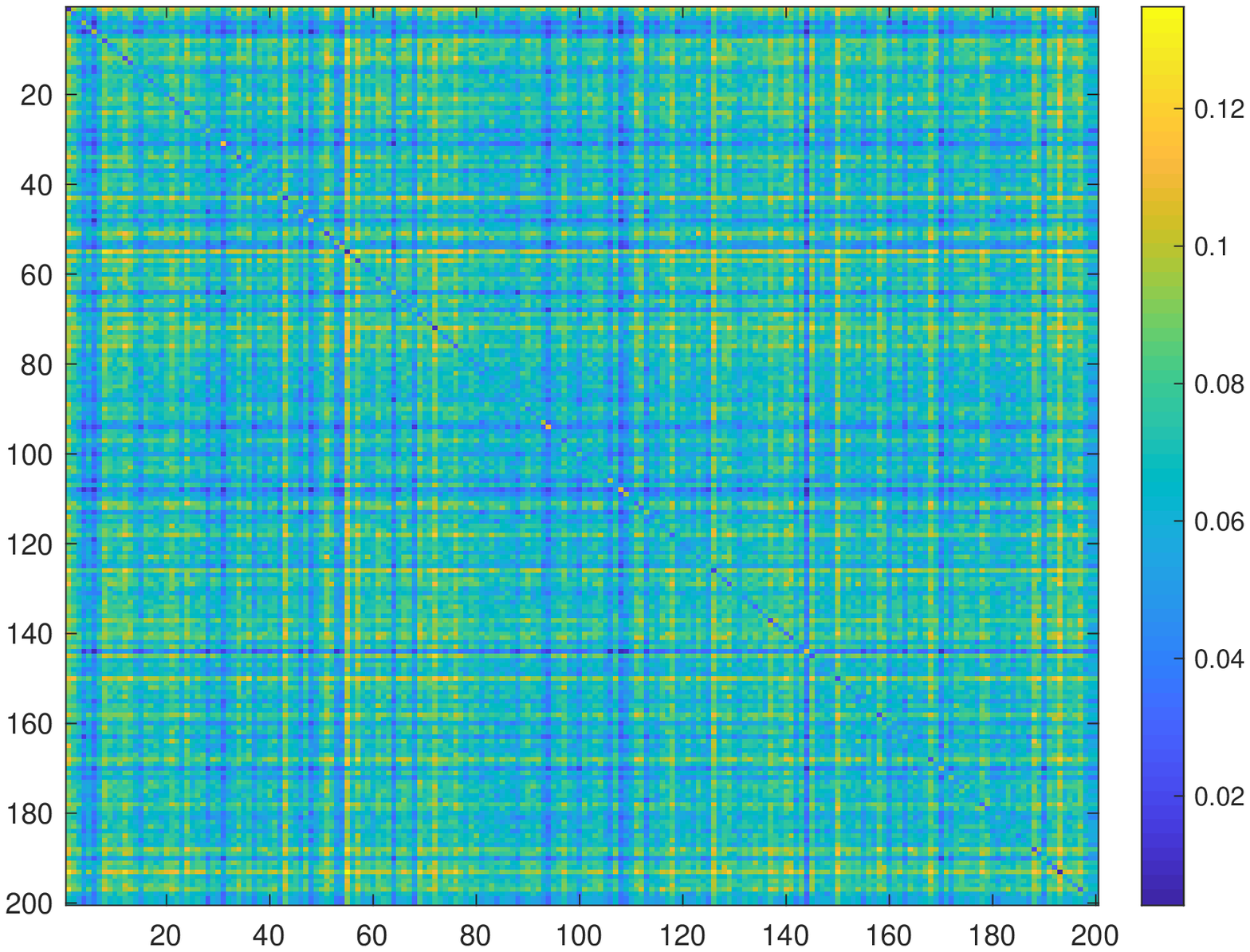}
\end{tabular}
\caption{Comparison between GFIM and EFIM. Left: the GFIM,  Middle: the 1st diagonal block of EFIM, Right: the difference between them.}
\label{EFIM-GFIM-random}
\end{figure}
 \subsection{Constructions of the Curvature Matrices} 
The curvature matrices can be constructed in other ways. We accumulate the statistics by momentum or use the mini-batch gradients. For simplicity, we do not write the multi-step strategy explicitly here. 

\textbf{Accumulate Statistics of Matrix by Momentum}

Alternatively, we can replace the subsampled matrices \eqref{left-curvature-matrix} and \eqref{right-curvature-matrix} by:
  \be
  \label{accu-curvature-mat}
\begin{aligned} 
\overline{L}_k & =\beta \overline{L}_{k-1} + (1-\beta ) L_k, \\
\overline{R}_k & = \beta \overline{R}_{k-1} +(1-\beta ) R_k, 
\end{aligned}
\ee
  where $\overline{L}_0 =L_0$, $\overline{R}_0 = R_0$ and $\beta$ is the momentum parameter. 
  
\textbf{Construct Matrix by Mini-Batch Gradients}

At each iteration, the mini-batch gradient $\G_{B_k}$ is used to update the curvature matrices $\widetilde{L}_k$ and $\widetilde{R}_k$:
\be
\label{mini-curvature-mat}
\begin{aligned} 
\widetilde{L}_k & =\beta \widetilde{L}_{k-1} + \gamma \G_{B_k} \G_{B_k}^\top,\\% \label{left-curvature-matrix-o} \\
\widetilde{R}_k & = \beta \widetilde{R}_{k-1} + \gamma \G_{B_k}^\top \G_{B_k}, %\label{right-curvature-matrix-o}
\end{aligned}
\ee
  where $\widetilde{L}_0 =\lambda_0 I $, $\widetilde{R}_0 =\lambda_0 I$, $\beta$
  and $\gamma$ are the parameters for momentum. The mechanism reuses the mini-batch gradients and can be regarded as an online update of GFIM.

\subsection{Comparisons with Related Works}
The most related approaches are KFAC and Shampoo. To make the difference
clearer, we take a fully-connected layer ($m < n$) as an example. Given a sample
set $S_k$, the directions of three methods are listed as below:
\be D_k =\left \{ \begin{array}{ll} -(\lambda_k I + L_k)^{-1} \FOM_k , &\text{NG+} ,\vspace{1ex}\\ 
	 -(\widehat L_k)^{-1/4} \FOM_k (\widehat R_k)^{-1/4}, &\text{Shampoo} ,\vspace{1ex}\\ 
	 - (\widetilde A_k)^{-1}\FOM_k (\widetilde G_k)^{-1},  &\text{KFAC},\vspace{1ex}\\  
	 \end{array} \right. \ee
where $\widetilde G_k = \sqrt{\lambda_k} I + \frac{1}{|S_k|}\sum_{i\in S_k} g_i g_i^\top$ and $\widetilde A_k =  \sqrt{\lambda_k} I + \frac{1}{|S_k|}\sum_{i\in S_k} a_i a_i^\top$. $\widehat L_k$ and $\widehat R_k$ in Shampoo are exactly $\widetilde{L}_k$ and $\widetilde{R}_k$ in \eqref{mini-curvature-mat} by choosing $\beta=\gamma=1$.

Assume that the matrix update frequency of all three methods are the same. NG+
needs to compute the inverse of $L_k$ and a matrix-matrix multiplication. The
computational cost of KFAC is more than twice that of our method. Shampoo needs
to update two matrices and obtain the -1/4 inverse of both two matrices whose
computational cost is much more than that of computing the inverse. Although a
coupled Newton method is used in \cite{anil2020second}, the implementation is more complicated and the end-users need to tune more hyper parameters. 

In the distributed setting, KFAC has to synchronize two matrices with the size
$m^2$ and $n^2$ while NG+ usually needs to synchronize $L_k$ whose size  $\min
\{m^2, n^2\}$ is smaller. When the matrix of NG+ is chosen by
\eqref{mini-curvature-mat}, we do not need synchronize $L_k$ explicitly. KFAC
and Shampoo store two different matrices and the storage is $m^2 + n^2$.
However, NG+ only stores one matrix and the storage is $\min \{m^2, n^2\}$. This
means the extra memory overhead of NG+ is smaller than the memory overhead of
storing a tensor with the same size as gradient. It leads to an advantage in certain cases, for example, when one of the dimensions of embedding matrix may be 10 million or even more.

\subsection{Computational Complexity}
We next show the detailed extra computational operations compared with SGD.  
At each iteration,  the construction of $L_k$, the inversion $L_k^{-1}$, and a
matrix multiplication of $L_k^{-1}\mathbf{G}_k$ need to be computed. The
computational cost of the first two operations can be controlled by the
multi-step strategy. The last matrix multiplication is embarrassingly fast in
GPUs. Note that in practice, we can simply let the sample set $S_k = B_k$ when constructing $L_k$.
%The extra storage consumption is the matrix $L_k$. However, the size is $\min\{m^2,n^2\}$ and is usually much smaller than the gradients.

The computation of $L_k$ for the cases \eqref{accu-curvature-mat} and \eqref{mini-curvature-mat} is a little bit different. For the case \eqref{accu-curvature-mat}, we need to use the gradients of each sample. Note that commonly the deep learning framework such as PyTorch yields mini-batch gradient. Luckily, the required information $\{G_i,A_i\}$ are already computed and can be stored in the process of computing the mini-batch gradient. In distributed cases with $\mathcal{M}$ devices, the sample set is split by $S_k = [S_{k,1},\dots, S_{k, \mathcal{M}}]$. We compute matrix $L_{k,j}  =  \frac{1}{|S_{k,j}|}\sum_{i\in S_{k,j}} \G_{i,k}\G_{i,k}^\top$ with sample slices $S_{k,j}$ and then average (All-Reduce) them among all devices. The synchronous cost can be also controlled  by the delayed update strategy. Since the case \eqref{mini-curvature-mat} uses the mini-batch gradient, we do not need extra distributed communication operations.
%We summarize the above statements in Table \ref{Extra-computation}.

%\input{./sections/acceleration}

\section{Efficient Computation of the GFIM Direction}
% to Reduce Computational Complexity}

\subsection{Utilize the Structure of the Gradients}
In this part, we show how to save computational cost by taking full use of the structures of the gradients. Take the case $m \leq n$ as an example. In neural networks, the gradient is computed by back propagation (BP) process and often has the format \eqref{Grad-Struc} where $G_i$ and $A_i$ are obtained from backward and forward processes, respectively. 
Specifically, $\kappa$ in fully-connected layers and embedding layers equals 1 while $\kappa$ in convolutional layer is larger. 
Unluckily, $m$ and $n$ can be both very huge in applications. Hence, the cost of computing the inversion might be expensive. 

%\subsubsection{$\kappa$ = 1}
 When $\kappa$ =1, $g_{i,k}$ and $a_{i,k}$ are actually vectors and $L_k$ degenerates to be $\lambda_k I$ plus a low-rank matrix: 
  
\bee\label{kappaequalsone} 
   L_k = \lambda_k I_{m\times m} + \frac{1}{|B_k|} 
   \sum_{i\in B_k} c_{i,k} g_{i,k}g_{i,k}^\top, 
\eee
   where $c_{i,k} = a_{i,k}^\top a_{i,k}$ is a scalar. By using the Sherman-Morrison-Woodbury (SMW) formula, we obtain:
 \be
    \label{smw-fc}
   L_k^{-1} = \frac{1}{\lambda_k} I -  {U_k}\left(\lambda_k I + {U_k}^\top U_k\right )^{-1} U_k^\top,
\ee
where ${U_k} = \frac{1}{\sqrt{|B_k^{'}|}}\left [\sqrt{c_{1,k}}g_{1,k},\dots,\sqrt{c_{|B_k^{'}|,k}} g_{|B_k^{'}|,k} \right ]\in \Rn^{m\times |B_k^{'}| } .$ It is already shown in \cite{yang2020sketchy} that computing $(\lambda_k I + L_k)^{-1} \mathbf{G}_k$ by \eqref{smw-fc} is equivalent to a regularized least squares (LS) problems as follows:
\be\label{ls-fc} \min_{D\in \Rn^{m\times n}} \|U_k D - \mathbf{G}_k\|_F^2 + \lambda \|D\|_F^2.\ee
To reduce the computational complexity, we can instead solve the following sketching LS problems: 
\be\label{sketching-ls-fc} \min_{D\in \Rn^{m\times n}} \|\Omega_kU_k D - \Omega_k\mathbf{G}_k\|_F^2 + \lambda \|D\|_F^2,\ee
where $\Omega_k \in \Rn^{q\times m}, q\ll m$ is a sketching matrix. 

However, when $\kappa \neq $1, the summation part of $L_k$ is usually not a
low-rank matrix and  it is not reasonable to use SMW to reduce the computational complexity. Instead, we consider the following two strategies. Note that these strategies also work when $\kappa = 1$.

%\subsection{Streaming Methods to Compute the Inversion }

%\be\label{kappanotequalone} 
%\begin{aligned}
%   L_k &= \lambda_k I_{m\times m} + \frac{1}{|B_k^{'}|} \sum_{i\in B^{'}_k} G_{i,k}A_{i,k}^\top A_{i,k} G_{i,k}^\top \\
%   \end{aligned}
%   \ee
\subsection{Matrix Multiplications by Sketching}
The computational complexity of $\G_k^\top \G_k$ or $\G_k\G_k^\top$ is still not
tractable when both $m$ and $n$ are very huge. Here, we take sketching techniques to
reduce the complexity. Similar idea is already used in the Newton Sketch method
\cite{pilanci2017newton}. Given a sketching matrix $\Omega_k\in\Rn^{n \times q}$
where $E[\Omega_k \Omega_k^\top] = I_n$ and $q\ll n$, the matrix  $\G_k\G_k^\top $
 can be approximated  by
\[
\G_k\G_k^\top \approx \left (\G_k\Omega_k \right )  \left (\G_k\Omega_k \right )^\top.
\]
Here, we consider random uniform row samplings. Specifically, each column of
$\Omega_{:,i}, \; i=1,2,\dots, q,$ is sampled from \be\label{sketch-mat}
\omega\leftarrow  \sqrt{\frac{n}{q}} e_i, %\frac{e_{i}}{\frac{1}{n}},
i=1,2,\dots,n,\ee 
with probability $\frac{1}{n}$, where $e_i\in \Rn^n$ is the vector whose $i$-th
element is 1 and 0 otherwise. Note that sketching the matrix $\G_k$ by uniform row sampling does not involve extra computation and can be finished easily.
 There are also other different sketchy ways. 

%this makes the algorithm as simple as possible.

\subsection{Block-Diagonal Approximation to the Curvature Matrix}
We can use a diagonal matrix to approximate $L_k$ or $R_k$. Assume that $m = sp$ and denote $(\G_k)_{(i-1) p+1:ip, :}$ by $\G_k^{i,p} $, where $i=1,2,\dots,s$. The approximation is presented as follows:
\be
\label{diagonal-appro}
\G_k \G_k^\top \approx \diag \{ \G_k^{1,p}(\G_k^{1,p})^\top, \dots,\G_k^{s,p}(\G_k^{s,p})^\top \}.
\ee
Since the size of each $\G_k^{i,p}(\G_k^{i,p})^\top$ is the same, computing the inverse of $\G_k^{i,p}(\G_k^{i,p})^\top$ can be done in a batched fashion, which is significantly faster than inverting each matrix individually. By using \eqref{diagonal-appro}, the computational cost reduces from $O(m^3)$ to $O(sp^3)$.

\section{ Theoretical Analysis}
\subsection{Global Convergence}
In this part, the convergence analysis of NG+ is established. We make the following standard assumptions in stochastic optimization. Assume that $m \leq n$ and $\mathcal{T}$ is a constant just for simplicity.\begin{assumption} \label{asp: global convergence}
We assume that $\Psi(\cdot)$ satisfies the following conditions.
\begin{enumerate}
\item 
$\Psi(\Theta)$ is continuous differentiable on $\mathbb{R}^{m\times n}$ and has a lower bound, i.e., $\Psi(\Theta)\geq \Psi^*$ for any $\Theta$. The gradient $\nabla \Psi(\Theta)$ is ${L}_\Psi$-Lipschitz continuous, i.e.
$\|\nabla \Psi(\Theta_1)-\nabla \Psi(\Theta_2)\|_F\leq {L}_\Psi \|\Theta_1-\Theta_2\|_F. $
\item There exists positive constants $h_1$, $h_2$ such that $h_1 I \preceq L_k + \lambda_k I \preceq h_2 I$ holds for all $k$.
\item The mini-batch gradient is unbiased a.s. $\nabla \Psi(\Theta_k)=\mathbb{E}[G_k|\Theta_k,\cdots,\Theta_0]$ and has bounded variance $\mathbb{E}[\|G_k-\nabla \Psi(\Theta_k)\|_F^2|\Theta_k,\cdots,\Theta_0]\leq\sigma_k^2$ for all $k$.
\end{enumerate}
\end{assumption}
These assumptions are broadly used in second-order methods, such as  \cite{yang2020S2QN, goldfarb2020practical}.
We now present the global convergence guarantee of NG+.

\begin{theorem}
Suppose that Assumption \ref{asp: global convergence} is satisfied and the step size $\{\alpha_k\}$ satisfies $\alpha_k\leq\frac{2h_1^2}{L_\Psi h_2}$, $\sum\alpha_k=\infty$ and $\sum\alpha_k^2\sigma_k^2< \infty$. Then with probability 1 we have
$$\lim_{k\rightarrow\infty}\|\nabla\Psi(\Theta_k)\|_F=0.$$
\end{theorem}

We next consider the complexity of NG+ and the following theorem implies that $O(\epsilon^{-\frac{1}{\beta}})$ iterations are enough to guarantee that $\frac{1}{\hat{T}}\sum_{k=1}^{\hat{T}}\mathbb{E}[\|\Psi(\Theta_k)\|^2]\leq\epsilon$.

\begin{theorem}
Suppose that Assumption \ref{asp: global convergence} is satisfied and the step size is chosen as
$\alpha_k=\frac{2h_1^2}{L_\Psi h_2}k^{-\beta},$
where $\beta\in(0.5,1)$ and $\sigma_k \equiv \sigma$ for all $k$. Then we have
$$\frac{1}{\hat{T}}\sum_{k=1}^{\hat{T}}\mathbb{E}[\|\Psi(\Theta_k)\|^2]\leq \frac{L_\Psi h_2^2}{h_1^2}\hat{T}^{-1}+\frac{2\sigma^2}{h_2^2(1-\beta)}\hat{T}^{-\beta},$$
where $\hat{T}$ is the number of iterations. \end{theorem}
The proof of the above two theorems can be found in the Appendix.

\subsection{Regret Analysis}

In this part, we consider the regret bound of one variant of NG+ under standard online convex optimization setting. The regret is defined as follows: 
\bee \mathcal{R}_T
= \sum_{t=1}^T  \psi_t(\Theta_t)  
    - \inf_{\Theta^* \in \mathcal{K}} \left (\sum_{t=1}^T  \psi_t(\Theta^*)  \right),\eee
 where $\psi_t:\mathcal{K}\rightarrow \mathbb{R}$ is a convex cost function,  $\Theta_t\in \mathcal{K}$ and $\mathcal{K}$ is a bounded convex set, i.e., $\forall X,Y\in \mathcal{K}$, we have $\|X-Y\|_F \leq \mathcal{D}$.
We analyze the following iteration process which can be regarded as an extension of the online Newton method \cite{hazan2007}:
\begin{equation}\label{update1}
\begin{aligned}
    L_t =L_{t-1}+\nabla \psi_t^\top \nabla \psi_t, \quad 
    \Theta_{t+1} = \Pi_\mathcal{K}(\W_t - \frac{1}{\alpha} L_t^{-1} G_t),
    \end{aligned}
\end{equation}
where $L_0 = \epsilon I$, $\alpha$ is the step size. A few necessary assumptions are listed below.

\begin{assumption}
 \label{asp: matrix exp concave} We assume each $\psi_t(\cdot)$ satisfies the following conditions:
 \begin{enumerate} \item For any $t=1,2,\cdots$, the function $\psi_t$ satisfies:
$\psi_t(X)\geq \psi_t(Y)+\langle\nabla \psi_t(Y), X-Y\rangle+\frac{\alpha}{2}\|X-Y\|_{\nabla \psi_t(Y)\nabla \psi_t(Y)^\top}^2,$
for $X,Y \in \mathcal{K},$ where $\|X-Y\|_{\nabla \psi_t(Y)\nabla \psi_t(Y)^\top}^2 = \mathrm{Tr}\left ((X-Y)^\top\nabla \psi_t(Y)\nabla \psi_t(Y)^\top(X-Y)\right).$
\item The norm of the gradient $\nabla \psi_t$ is bounded by $\mathcal{L}_G$, i.e., $\|\nabla \psi_t\|_F \leq \mathcal{L}_G$.
\end{enumerate}
\end{assumption}
When $n=1$, the function satisfying Assumption 4.1 is actually the $\alpha$-exp-concave function \citep{hazan2007logarithmic}. We summarize our logarithmic regret bound as follows.
\begin{theorem}
Let $\epsilon=\frac{2}{\alpha \mathcal{D}^2}$. If Assumption \ref{asp: matrix exp concave} is satisfied, then the regret $\mathcal{R}_T$ can be bounded by
$$\mathcal{R}_T \leq \frac{n}{\alpha}\log \alpha \mathcal{L}_G^2\mathcal{D}^2T.$$
\end{theorem}
The proof is shown in the Appendix.

\section{Numerical Experiments}

\subsection{Image Classification}

The training of ResNet50 \cite{he2016deep} on ImageNet-1k
\cite{deng2009imagenet} dataset is one of the basic experiments in image classification \cite{mattson2019mlperf}.
%The numerical performance on the task is an important criterion to evaluate the new proposed algorithms.
We compare NG+ with LARS \cite{you2017large}, SGD with momentum (SGD for short)
and KFAC.  The experiments of LARS are based on the
implementation\footnote{\url{https://github.com/NUS-HPC-AI-Lab/LARS-ImageNet-PyTorch}}
and tuned in the same way as Table 3 in \cite{you2019large}. SGD is taken
from the default version in PyTorch. 
The
result of ADAM is not reported because it does not perform well in ResNet50 with
ImageNet-1k task. We do not
report the results of Shampoo since an efficient implementation of Shampoo in
PyTorch is not officially available and our current codes do not take advantage
training on heterogeneous hardwares. We
follow the standard settings and use the same basic dataset augmentation as the
official PyTorch
example %\footnote{\url{https://github.com/pytorch/examples/tree/master/imagenet}}
\textbf{without changing the figure sizes throughout the training process.}  The training
process of all  methods are terminated  once the top-1 testing accuracy equals
or exceeds 75.9\%. All  codes are written in PyTorch and available in \url{https://github.com/yangorwell/NGPlus}.%the supplementary materials. % Note that the mechanisms of distributional communication and datasets loading are the same.  %The volatility of the performance of first-order methods are lower among different hardwares and softwares.

\begin{table}[htb]
\centering
\caption{Detailed Statistics on image classification when top-1 testing accuracy achieves 75.9\%. }%s means seconds while h means hours.}
%{\color{red} To be completed.}.}
{
%\vspace{0.5ex}
\begin{tabular}{c|ccc}
\hline \hline
&\# Epoch & Total Time  &Time Per Epoch \\ \hline
    SGD & 76 & 10.9 h  & {517 s} \\ \hline
LARS &74 & 10.7 h & {521 s} \\\hline
KFAC & 42 & 8.0 h & 686 s\\ \hline
NG+ & 40& 6.7 h & 600 s\\\hline
  \hline
\end{tabular}
}
\label{ImageNet-stat}
\end{table}

We first consider a batch size $256$. Note that the SGD achieves top-1 testing
accuracy 75.9\% within 76 epochs, which is well tuned. The changes of testing
accuracy and training accuracy versus training time are reported in Figure
\ref{imagenet-comparison} and detailed statistics are shown in Table
\ref{ImageNet-stat}. We can see that NG+ performs best in the total training
time and only takes 40 epochs to reach 75.9\% top-1 testing accuracy. Compared
with SGD and LARS, NG+ has at least 45\% and 37\% reduction in the number of
epochs and training time, respectively. Although NG+ is better than KFAC for
only two epochs, it leads to 16\% reduction in terms of the computing time.

\begin{figure}
\caption{Numerical performance on ResNet50 on ImageNet-1k.
}
\centering
\vspace{0.5ex}
\begin{tabular}{cccc}
\includegraphics[width=0.3\textwidth]{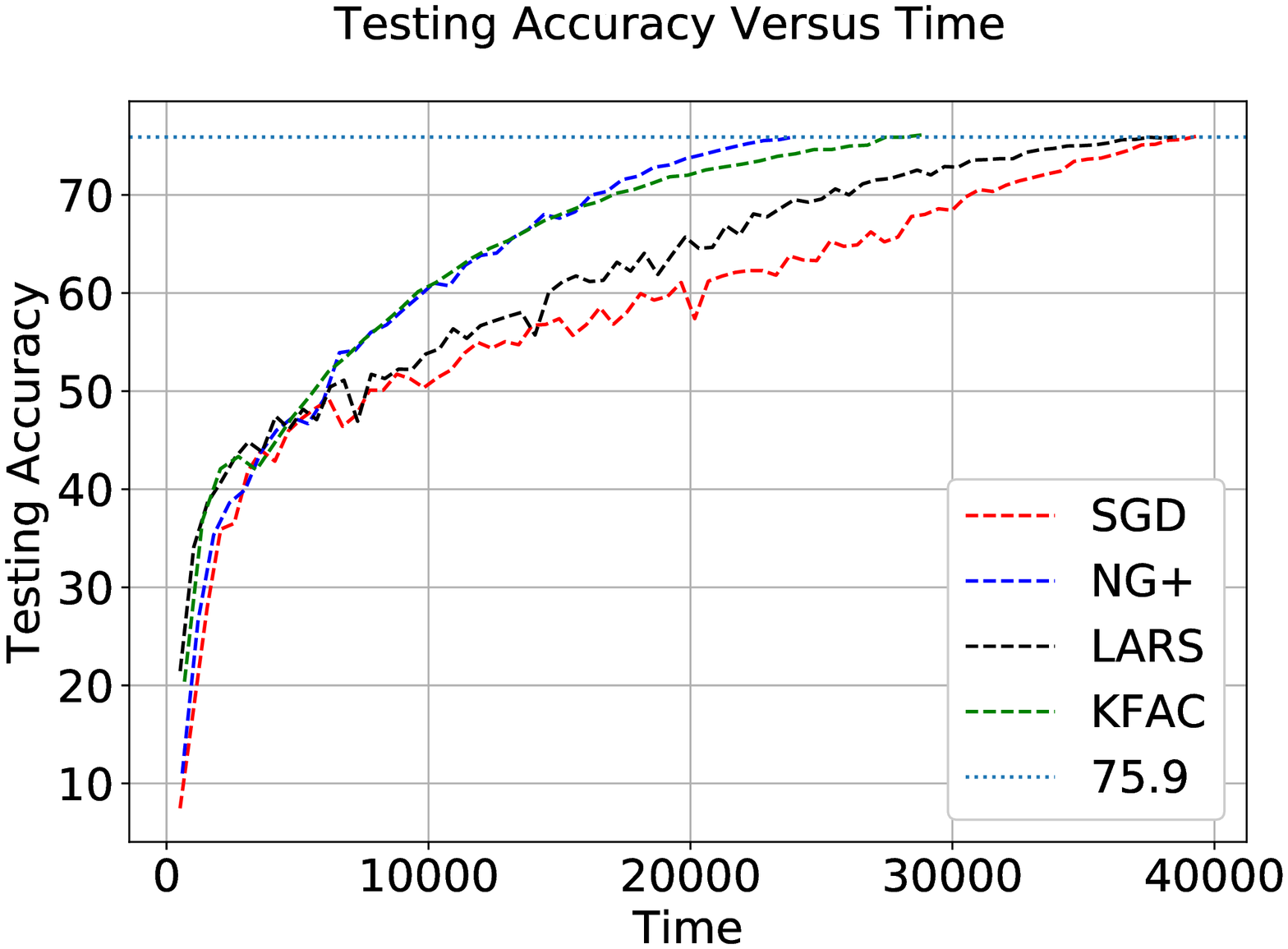}
\includegraphics[width=0.3\textwidth]{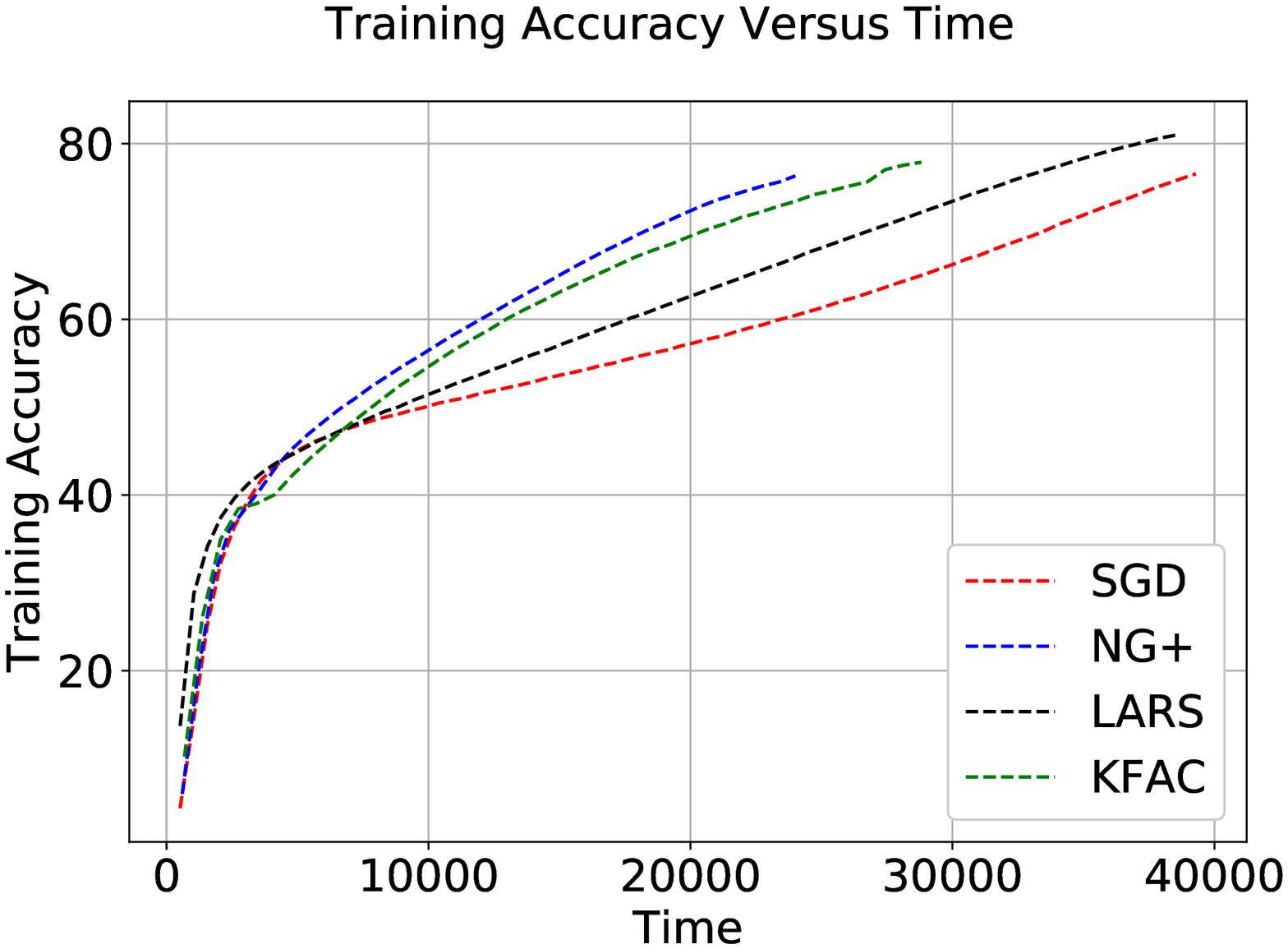}\\
\end{tabular}
\label{imagenet-comparison}
\end{figure}

\begin{figure}[ht]
\caption{Numerical performance of SchNet on QM9 and MatProj.}
\centering
\vspace{0.5ex}
\begin{tabular}{c}
\subfigure[QM9: Loss VS Time.] 
{\includegraphics[width=0.3\textwidth]{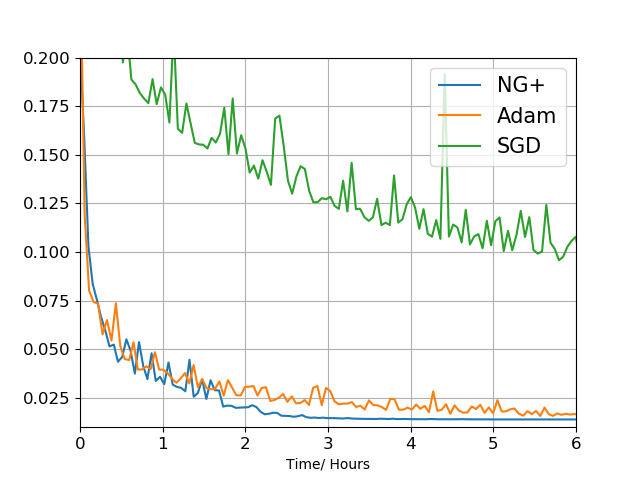}}
\subfigure[QM9: Loss VS Epoch.] {
\includegraphics[width=0.3\textwidth]{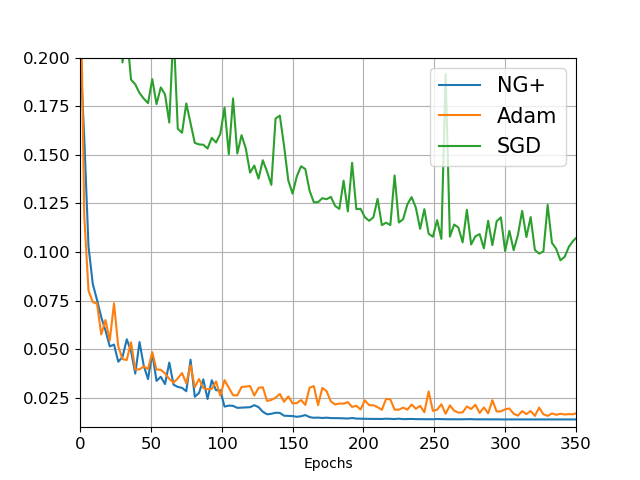}}\\
\subfigure[MatProj: Loss VS Time.] {
\includegraphics[width=0.3\textwidth]{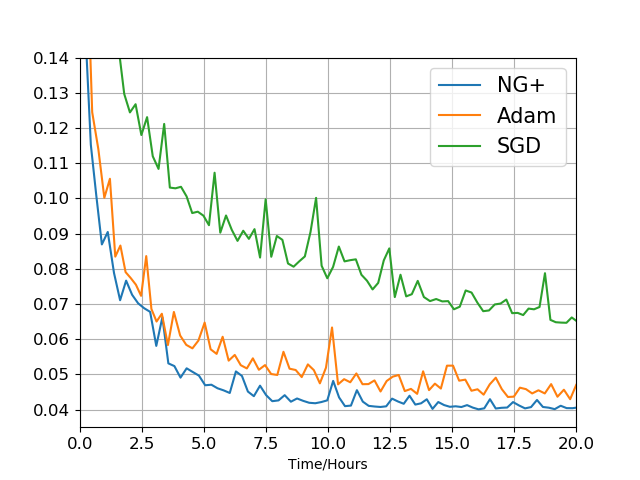}}
\subfigure[MatProj: Loss VS Epoch.] {
\includegraphics[width=0.3\textwidth]{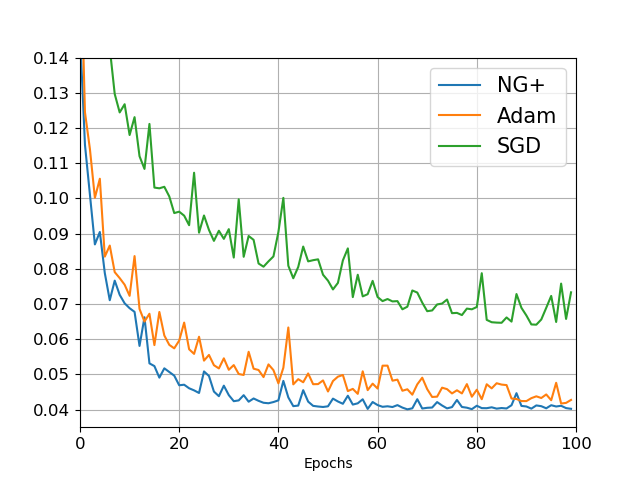}}
\end{tabular}
\label{qm9-performance}
\vspace{-2ex}
\end{figure}

We further consider the large-batch training. Since our GPUs are limited,  we
accumulate mini-batch gradients sampled in several steps to obtain the gradient
of a larger batch. Hence, only the statistics on the epochs and iterations are reported and can be found in
Table \ref{ImageNet-large-bs-stat}. The
number of epochs for the batch sizes 2048 and 4096 is 41. By running more experiments and
finding more reliable tuning strategies, it is expected that the number of epochs can be further reduced for the large-batch setting.

\begin{table}[htb]%[H]%[!b]
%\scriptsize
\centering
{
\caption{Detailed Statistics of NG+ on ImageNet-1k for different Batch Sizes.}
\begin{tabular}{c|ccc}
\hline \hline
\# Batch Size& 256 & 2048  &4096  \\ \hline
\# Epoch& 40 & 41 & 41 \\ \hline
\#Iteration & 200k & 25.6k&12.8k \\ \hline
  \hline
\end{tabular}
\label{ImageNet-large-bs-stat}
}
\vspace{-2ex}
\end{table}

\subsection{Quantum Chemistry}
Deep learning has been applied to quantum chemistry problems. SchNet is a
well-known network architecture in quantum chemistry modeling
\cite{schutt2017schnet}. We compare NG+ with SGD and Adam \cite{kingma2014adam} by using SchNetPack
\footnote{\url{https://schnetpack.readthedocs.io/en/stable/index.html}} on the
benchmark datasets QM9 \cite{ramakrishnan2014quantum} and Materials Project (MatProj)
\cite{jain2013commentary}. The numerical results are presented in Figure
\ref{qm9-performance} where we can see the validation loss of NG+ is much better
than Adam in terms of both iteration and time on QM9 and MatProj datasets. For
example, on MatProj, when the validation loss attains 0.05, NG+ needs
less than 20 epochs and spends about 4 hours. On the other hand, Adam needs more
than 30 epochs and consumes about 7.5 hours.

\subsection{Neural Machine Translation}
We consider the neural machine translation task in this part. The implementation is based on the fairseq \cite{ott2019fairseq} and we use default transformer ``iwslt\_de\_en\_v2'' architecture therein. The IWSLT14 \cite{cettolo2014report} German-to-English dataset is used.  We present the training loss and validation loss through the training process in Figure \ref{transformer}. In the training set, the training loss curve of NG+ is always below that of Adam. We also report the BLEU score in the testing set. The result (mean and standard deviation) over 3 independent runs of NG+ is $34.79\pm0.09$ while the BLEU score of Adam is $34.75\pm0.08$.

\begin{figure}[ht]
\caption{Numerical performance on IWSLT14 (Transformer).
}
\centering
\begin{tabular}{cccc}
\includegraphics[width=0.3\textwidth]{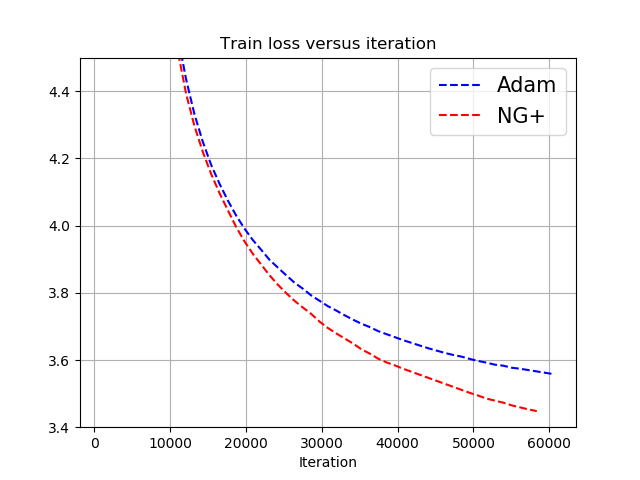}
\includegraphics[width=0.3\textwidth]{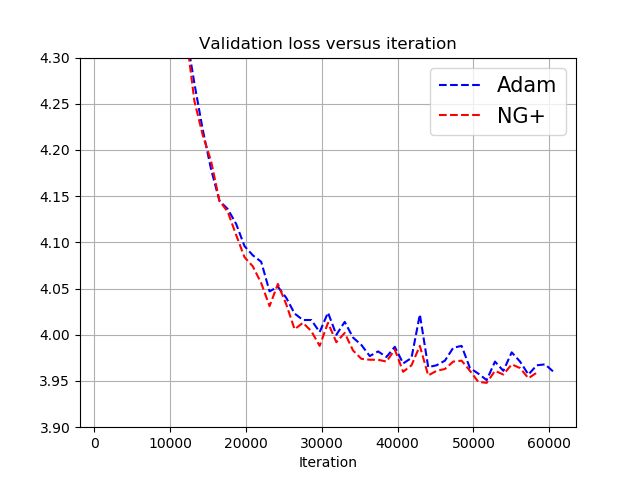}\\
\end{tabular}
\label{transformer}
\end{figure}
\subsection{Recommendation System}
We test the performance of NG+ on the recommendation models DLRM
\cite{naumov2019deep} which is widely considered in industry. The Criteo Ad
Kaggle dataset
\footnote{\url{https://labs.criteo.com/2014/02/kaggle-display-advertising-challenge-dataset/}}
in this part contains nearly 45 million samples in one week. The data
points of the first six days are used for training set while the others are used
for testing set. We train the model for one epoch, compare NG+ with SGD, AdaGrad \cite{duchi2011adaptive} and report the Click Through
Rate (CTR) on the training and testing set in Figure \ref{dlrm}. The best
testing accuracy of AdaGrad is 79.149\% while that of NG+ is 79.175\%. Our
proposed method has at least 0.025\% higher testing accuracy which is a good progress in this topic \cite{wang2017deep}.  

\begin{figure}[htbp]
\caption{Numerical performance on DLRM model.
}
\centering
\begin{tabular}{cccc}
\includegraphics[width=0.3\textwidth]{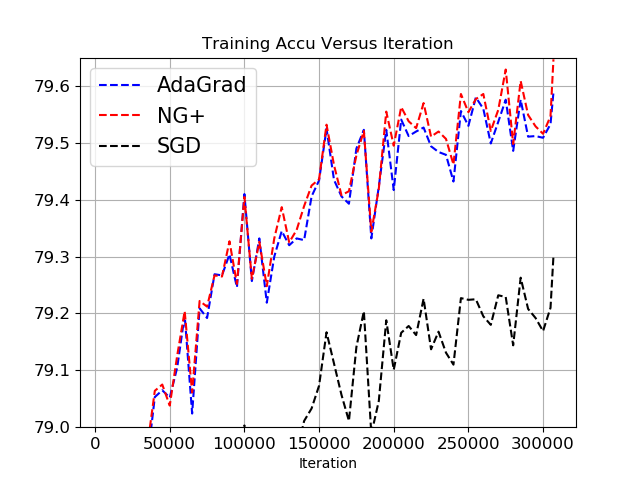}
\includegraphics[width=0.3\textwidth]{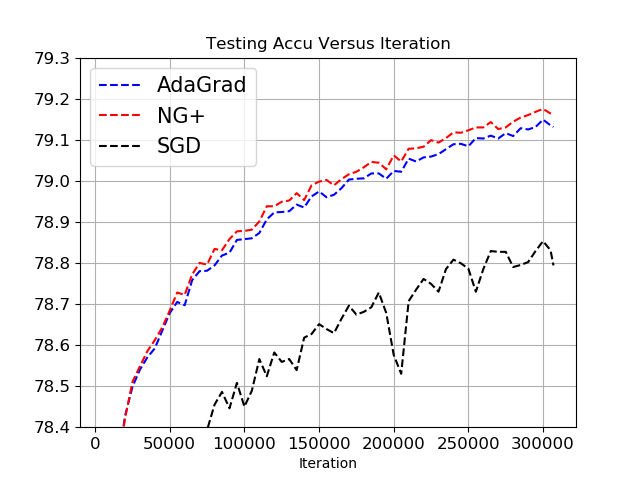}\\
\end{tabular}
\label{dlrm}
\end{figure}

\section{Conclusion}
In this paper, we propose a multi-step matrix-product natural gradient method
NG+. It is based on an intuitive extension of the Fisher information matrix to
the matrix space. % the GFIM which extend the FIM in vector space intuitively to the matrix space.
The size of the coefficient matrix to be inverted is
reasonably small and the computational cost is reduced comparing to the
state-of-the-art methods including KFAC and Shampoo. %We further consider the multi-step strategy and other efficient strategies so that the computational complexity can be reduced. 
The global convergence property is analyzed under
some mild conditions and a regret bound is given in the online convex
optimization case. Numerical experiments on important tasks, such as image
classification with ResNet50, quantum chemistry modeling with Schnet, neural
translation with Transformer and recommendation system with DLRM, show that NG+
is quite promising.

The performance of NG+ can be further improved in several aspects, including
accelerating the training process and improving the testing accuracy, with the help of a number of techniques such
as sketching, more reliable learning rates strategies and training on  heterogeneous  hardware. Implementing the fundamental operations such
as the computing of sample gradients and matrix inversion
in a more efficient and friendly way is also critical to second-order type methods.
 A few particularly
important topics of investigation are (i) a comprehensive study of the issue of
larger batch sizes, 
(ii) a careful design for the parallel efficiency, and (iii)
 extensive experiments on various large learning tasks.

%In the future, we prepare to test NG+ on the task ResNet50 + ImageNet-1k by using larger batch size. 

%NG+ has more operations such as inverting a matrix, computing sample gradient and etc.. Therefore it is more complex than first-order method. Although we use different strategies to reduce the computational complexity, it is expected that the implementation can be further improved and optimized if more flexible API supports from deep learning framework such as PyTorch are available.

\bibliographystyle{plain}
\bibliography{reference}

\newpage
\appendix
\section{Implementation Details}
In this section, we describe the implementation details of the numerical experiments.

%We next describe the tuning schemes of hyper-parameters. The learning rate is very important for performance and we mainly consider the following  two schemes.
\begin{itemize}
\item ResNet50 on ImageNet-1k: 
\begin{itemize}
\item We first consider the case where the batch size is 256.
We use the official implementation of ResNet50 (also known as ResNet50 v1.5) in PyTorch. The detailed network structures can be found in the website: \url{https://pytorch.org/docs/stable/_modules/torchvision/models/resnet.html}. We use the linear warmup strategy \cite{goyal2017accurate} in the first 5 epochs for SGD, KFAC and NG+. 
\begin{itemize}
\item SGD uses the cosine learning rate strategy \[ \alpha_k = 0.001 + 0.5 * (\alpha_0 - 0.001) * (1 + \text{cos(epoch\_k / max\_epoch} * \pi))\footnote{Note that $k$ is the number of the iteration while epoch\_k means the number of passing the whole dataset. epoch\_k = $k$/(\#dataset/batch\_size).} \] for given parameters max\_epoch and the initial learning rate $\alpha_0$. The hyper-parameters are the same as those on the website \footnote{\url{https://gitee.com/mindspore/mindspore/blob/r0.7/model_zoo/official/cv/resnet/src/lr_generator.py}} and the result reported here achieves top-1 testing accuracy 75.9\% within 76 epochs. This is better than the result of using the diminishing learning rate strategy in official PyTorch example\footnote{\url{https://github.com/pytorch/examples/tree/master/imagenet}}, which needs nearly 90 epochs.
\item LARS uses the codes from the website\footnote{\url{https://github.com/NUS-HPC-AI-Lab/LARS-ImageNet-PyTorch/blob/main/lars.py}} and we use the recommended parameters from \cite{you2019large}.
\item KFAC uses the exponential strategy \[ \alpha_k = \alpha_0 * (1 - \text{epoch\_k/max\_epoch)}^{\mathcal{\widetilde D}},\] for given the max\_epoch, decay rate $\mathcal{\widetilde D}$ and the initial learning rate $\alpha_0$. The max\_epoch is set to be 60. The initial learning rate is from \{0.05, 0.1, 0.15\} and initial damping $\lambda_0$ is from  \{0.05, 0.1, 0.15\} and we report the best results among them. The damping is set to be $\lambda_0\cdot (0.87)^{\text{epoch\_k}/10}$. The update frequency which means $\mathcal{T}$ in multi-step strategy is 500.
\item NG+ uses the exponential strategy by letting max\_epoch to be 52, the initial learning rate to be 0.18 and decay\_rate to be 5, respectively. The damping parameter $\lambda_k$ is set to be 0.16$\cdot (0.8)^{\text{epoch\_k}/10}$. The update frequency is 500.
%\item Note that both SENG and KFAC are not sensitive to damping and initial learning rate. The weight decay for SENG and KFAC are chosen the best from \{5e-4, 3e-4, 2e-4, 1e-4\}. We also consider the cosine strategy for KFAC and SENG, but it does not work well.
\end{itemize}
\item For large-batch training, the detailed hyper-parameters for each batch-size are listed in Table \ref{large-bs-table}.
\begin{table}[h]
\centering
\begin{tabular}{ccccccc}%{|p{10ex}p{8ex}p{8ex}|}
\toprule
Batch Size & $\alpha_{\text{warmup}}$ & warm-up epoch & $\alpha_0 $ & decay\_rate &max\_epoch & damping $\lambda$ \\ %&sp. (\%) \\
\midrule
2, 048 &0.01&5&1.3 &5&50&0.2$\cdot (0.7)^{(\text{epoch\_k}/10)}$\\
4, 096 &0.01&5&2.6 &5&50&0.3$\cdot (0.7)^{(\text{epoch\_k}/10)}$\\
\bottomrule
\end{tabular}
\vspace{1ex}
 \caption{Detailed hyper-parameters for different batch sizes.}
\label{large-bs-table}
\end{table}
\end{itemize}
\item SchNet:

%We use the code of and change the optimizer to our proposed NG+.
We use SchNetPack\footnote{\url{https://github.com/atomistic-machine-learning/schnetpack} }  and implement our proposed NG+ within this package.
The weight decay is set to be 1e-5 for all the experiments. 
To train the model efficiently, it is highly recommended to put all the datasets on an Solid State Disk(SSD).
\begin{itemize}
\item For NG+, we use the learning rate schedule $\alpha_k= \text{2e-4} \cdot \text{max}(\text{epoch\_k} - 10, 1)^{-0.5}$. Damping is 0.17 for QM9, 0.6 for MatProj. The update frequency is set to 200.
\item For Adam, we grid search the learning rate from {1e-5, 2e-5, 4e-5, 1e-4, 2e-4, 4e-4} and choose 1e-4 for both MatProj and QM9. We also consider the schedule used in NG+ but it does not perform well in Adam. 
\item For SGD, we set the learning rate 4e-4 for QM9, 5e-3 for MatProj.
\end{itemize}
\item Transformer:

We use the fairseq\footnote{\url{https://github.com/pytorch/fairseq/}} and implement our proposed NG+ within this package. The batch size and the weight decay is set to be 4096 and 1e-4, respectively.
\begin{itemize}
\item For Adam, we set the initial learning rate to be 5e-4, eps to be 1e-8, $\beta_1$ to be $0.9$, $\beta_2$ to be $0.98$, the number of warmup-updates to be 10000 and the learning rate scheduler to be inverse\_sqrt, which is similar to the recommended example \footnote{\url{https://github.com/pytorch/fairseq/tree/v0.8.0/examples/translation}}.
\item For NG+, the initial learning rate is set to be \text{6e-4} and the damping is set to be $0.3\cdot0.87^{(\text{epoch\_k}/10)}$. Other parameters are the same as Adam.
\end{itemize}
\item DLRM:

We use the code in the website \footnote{\url{https://github.com/facebookresearch/dlrm}} and change the optimizer to our proposed NG+.
None of the optimizer applies the weight decay.
\begin{itemize}
\item For NG+, we set the learning rate 0.015, damping 1.0 and the update frequency 1000. 
\item For Adagrad, we set the learning rate 0.01. 
\item For SGD, we set the learning rate 0.1, with all other parameters identical to the file\footnote{\url{https://github.com/facebookresearch/dlrm/blob/master/bench/dlrm_s_criteo_kaggle.sh}}.
\end{itemize}
\end{itemize}

\section{Proof of Theorem 2}
\begin{proof}
We define $\mathbb{E}_k[\cdot]=\mathbb{E}[\cdot|\Theta_{k},\cdots,\Theta_0]$. By Assumption 1, it holds:
\be
\label{estimate-theo}
\begin{aligned}
    \mathbb{E}_k[\Psi(\Theta_{k+1})]&\leq \mathbb{E}_k\left[\Psi(\Theta_k)+\langle\nabla\Psi(\Theta_k),\Theta_{k+1}-\Theta_k\rangle+\frac{L_\Psi}{2}\|\Theta_{k+1}-\Theta_k\|_F^2\right]\\
    & = \Psi(\Theta_k) + \mathbb{E}_k\left[\langle\nabla\Psi(\Theta_k),-\alpha_k L_k^{-1}G_k\rangle\right]+\mathbb{E}_k\left[\frac{L_\Psi \alpha_k^2}{2}\|L_k^{-1}G_k\|_F^2\right]\\
    & \leq \Psi(\Theta_k) + \langle\nabla\Psi(\Theta_k),-\alpha_k L_k^{-1}\nabla\Psi(\Theta_k)\rangle+\frac{L_\Psi \alpha_k^2}{2h_1^2}\mathbb{E}_k\left[\|G_k\|_F^2\right]\\
    & \leq \Psi(\Theta_k) - \frac{\alpha_k}{h_2}\|\nabla\Psi(\Theta_k)\|_F^2+\frac{L_\Psi \alpha_k^2}{2h_1^2}\left(\|\nabla\Psi(\Theta_k)\|_F^2+\sigma_k^2\right)\\
    & \leq \Psi(\Theta_k) - \left(\frac{\alpha_k}{h_2}-\frac{L_\Psi \alpha_k^2}{2h_1^2}\right)\|\nabla\Psi(\Theta_k)\|_F^2+\frac{L_\Psi \alpha_k^2}{2h_1^2}\sigma_k^2.
\end{aligned}
\ee
Using $\alpha_k\leq\frac{2h_1^2}{L_\Psi h_2}$, summing over above inequality from $0$ to $\infty$, and taking the expectation yield
\begin{align*}
    \sum_{k=0}^{\infty}\frac{\alpha_k}{2h_2} \mathbb{E} \|\nabla\Psi(\Theta_k)\|_F^2\leq \Psi(\Theta_0) - \Psi^*+ \frac{L_\Psi}{2h_1^2} \sum_{k=0}^\infty \alpha_k^2\sigma_k^2 < \infty.
\end{align*}
Since $\sum_{k=0}^{\infty}\alpha_k=\infty$, we have
$\liminf_{k\rightarrow\infty}\|\nabla\Psi(\Theta_k)\|_F=0$ with probability 1.

We next prove $\lim_{k\rightarrow\infty}\|\nabla\Psi(\Theta_k)\|_F=0$ by contradiction. Suppose that there exists $\epsilon>0$ and two increasing sequences $\{a_i\}_{i=1}^\infty$ and $\{b_i\}_{i=1}^\infty$ such that $a_i<b_i$ and 
\begin{equation*}
   \left\{
\begin{aligned}
\|\nabla\Psi(\Theta_k)\|\geq2\epsilon, & & k=a_i, \\
\|\nabla\Psi(\Theta_k)\|\geq\epsilon, & & a_i<k<b_i, \\
\|\nabla\Psi(\Theta_k)\|\leq\epsilon, & & k=b_i.
\end{aligned}
\right.
\end{equation*}

Hence, it follows that
$$\epsilon^2 \sum_{i=1}^\infty\sum_{k=a_i}^{b_i-1}\alpha_k\leq \sum_{i=1}^\infty\sum_{k=a_i}^{b_i-1}\alpha_k\|\nabla\Psi(\Theta_k)\|^2\leq\sum_{k=1}^\infty\alpha_k\|\nabla\Psi(\Theta_k)\|^2\leq\infty,$$
which implies $\lim_{i\rightarrow\infty}\sum_{k=a_i}^{b_i-1}\alpha_k=0$.

In addition, we have
\begin{align*}
    \mathbb{E}\left[\sum_{k=a_i}^{b_i-1} \alpha_k^{-1}\|\Theta_{k+1}-\Theta_k\|_2^2\right] &= \mathbb{E}\left[\sum_{k=a_i}^{b_i-1} \alpha_k\|L_k^{-1}G_k\|_2^2\right] %\\
%    &\leq \mathbb{E}\left[\sum_{k=a_i}^{b_i-1} \frac{\alpha_k}{h_1^2}\|G_k\|_2^2\right] \\
%    &
    \leq \mathbb{E}\left[\sum_{k=a_i}^{b_i-1} \frac{\alpha_k}{h_1^2}\left(\|\nabla\Psi(\Theta_k)\|_2^2+\sigma_k^2\right)\right] < \infty.
\end{align*}

By H\"older inequality, we obtain
$$\|\Theta_{a_i}-\Theta_{b_i}\|_2^2 \leq \left(\sum_{k=a_i}^{b_i-1}\alpha_k\right)\left(\sum_{k=a_i}^{b_i-1} \alpha_k^{-1}\|\Theta_{k+1}-\Theta_k\|_2^2\right), $$
which implies $\lim_{i\rightarrow\infty}\|\Theta_{a_i}-\Theta_{b_i}\|_2^2=0$. However, we have $\|\nabla\Psi(\Theta_{a_i})-\nabla\Psi(\Theta_{b_i})\|_2\geq\epsilon$, which is a contradiction by the Lipschitz-continuous property of $\nabla \Psi(\Theta)$. This finishes the proof. \end{proof}

%\begin{theorem}
%Suppose that Assumption \ref{asp: global convergence} is satisfied and the step size is chosen as
%$\alpha_k=\frac{2h_1^2}{L_\Psi h_2}k^{-\beta},$
%where $\beta\in(0.5,1)$ and $\sigma_k \equiv \sigma$ for all $k$. Then we have
%$$\frac{1}{N}\sum_{k=1}^N\mathbb{E}[\|\Psi(\Theta_k)\|^2]\leq \frac{L_\Psi h_2^2}{h_1^2}N^{-1}+\frac{2\sigma^2}{h_2^2(1-\beta)}N^{-\beta},$$
%where $N$ is the number of iterations. This implies that $O(\epsilon^{-\frac{1}{\beta}})$ iterations are enough to guarantee that $\frac{1}{N}\sum_{k=1}^N\mathbb{E}[\|\Psi(\Theta_k)\|^2]\leq\epsilon$.
%\end{theorem}
\section{Proof of Theorem 3}
\begin{proof}
From \eqref{estimate-theo}, we have
$$\|\nabla\Psi(\Theta_k)\|^2\leq \frac{2h_2}{\alpha_k}\Psi(\Theta_k)-\frac{2h_2}{\alpha_k}\mathbb{E}_k[\Psi(\Theta_{k+1})]+\frac{L_\Psi \alpha_k}{h_1^2h_2}\sigma^2.$$

Taking expectation and summing over $k$ from $1$ to $N$, we obtain
\begin{align*}
    \sum_{k=1}^N\mathbb{E}[\|\nabla\Psi(\Theta_k)\|^2] &\leq \frac{2h_2}{\alpha_1}\mathbb{E}\Psi(\Theta_1)+\sum_{k=2}^N(\frac{2h_2}{\alpha_k}-\frac{2h_2}{\alpha_{k-1}})\mathbb{E}\Psi(\Theta_k)+\sum_{k=1}^N\frac{L_\Psi \alpha_k}{h_1^2h_2}\sigma^2\\
    &\leq \frac{2h_2}{\alpha_1}M+\sum_{k=1}^N\frac{L_\Psi\sigma^2}{h_1^2h_2}\alpha_k\\
    &\leq \frac{L_\Psi h_2^2}{h_1^2}+\frac{2\sigma^2}{h_2^2(1-\beta)}(N^{1-\beta}-1),
\end{align*}
which proves the theorem.
\end{proof}
\section{Proof of Theorem 5}
%If $n=1$, the  matrix version $\alpha$-exp-concave is the common vector version $\alpha$-exp-concave \citep{hazan2007logarithmic}. 
%If $\hat{f}(vec(X))$ is $\beta$-strongly convex and $\|\nabla f(X)\|_2\leq G$ for all $X\in\mathcal{K}$, by definition it is easy to show that $f(X)$ is $\beta/G^2$-exp-concave.

%\iffalse
We first give an example that satisfies the Assumption 4.1.

\textbf{Example}
Suppose that $\|X\|\leq 1$, $\|Y\|_2\leq R$ and $\mathcal{K}=\{W\in\mathbb{R}^{m\times n}|\|W\|_2\leq R \}$ for a given value $R$, then $f(W)=\frac{1}{2}\|W^\top X-Y\|_2^2$ satisfies Assumption 4.1.

\begin{proof}
With a slight abuse of the notation, we define $f(\text{vec}(W))=f(W)$.
We have $\nabla f(W)=XX^\top W-XY^\top$ and $\nabla^2 f(\text{vec}(W))=I\otimes XX^\top$. To show that $f(W)$ satisfies Assumption 4.1, we only need to prove that there exists a $\alpha$ such that for all $W$, we have $\nabla^2 f(\text{vec}(W))\succeq \alpha I\otimes \nabla f(W)\nabla f(W)^\top$.
\begin{align*}
    \nabla f(W)\nabla f(W)^\top&=X(X^\top W-Y^\top)(X^\top W-Y^\top)^\top X^\top\\
    &\preceq X 4R^2 I X^\top\\
    &=4R^2 XX^\top.
\end{align*}
Thus for $\alpha=\frac{1}{4R^2}$, we have $\nabla^2 f(\text{vec}(W))\succeq \alpha I\otimes \nabla f(W)\nabla f(W)^\top$.
\end{proof}
%\fi
To prove Theorem 5, we give two required lemmas in the next.
\begin{lemma} \label{lemma: regret bound}
Suppose that $\alpha\leq \alpha_0$ and Assumption 4 holds, then the regret can be bounded by
$$\mathcal{R}_T\leq\frac{1}{2\alpha}\sum_{t=1}^T \|G_t\|_{L_t^{-1}}^2 +\frac{\alpha\epsilon \mathcal{D}^2}{2}.$$
\end{lemma}
\begin{proof}
Let $\W^*\in\mathrm{argmin}_{\W\in\mathcal{K}}\sum_{t=1}^T\psi_t(\W)$ to be the minimizer in the hindsight. By Assumption 4, we have
$$\psi_t(\W_t)-\psi_t(\W^*)\leq \langle G_t,\W_t-\W^*\rangle-\frac{\alpha}{2}\|\W^*-\W_t\|_{G_tG_t^\top}^2.$$

As $\W_{t+1} = \Pi_\mathcal{K}(\W_t-\frac{1}{\alpha} L_t^{-1} G_t)$, we obtain
\begin{align*}
    \|\W_{t+1}-\W^*\|_{L_t}^2&\leq \|\W_t-\frac{1}{\alpha} L_t^{-1} G_t-\W^*\|_{L_t}^2\\
    &\leq \|\W_t-\W^*\|_{L_t}^2-\frac{2}{\alpha}\langle \W_t-\W^*,G_t\rangle+\frac{1}{\alpha^2}\|G_t\|_{L_t^{-1}}^2.
\end{align*}

Summing up over $t=1$ to $T$ yields
\begin{align*}
    &\sum_{t=1}^T \langle \W_t-\W^*,G_t\rangle \\ 
    =& \frac{1}{2\alpha}\sum_{t=1}^T \|G_t\|_{L_t^{-1}}^2 +\frac{\alpha}{2}\|\W_1-\W^*\|_{L_0}^2+\frac{\alpha}{2}\sum_{t=1}^T\|\W_t-\W^*\|_{L_{t}-L_{t-1}}^2-\frac{\alpha}{2}\|\W_{T+1}-\W^*\|_{L_T}^2 \\
    \leq 
    & \frac{1}{2\alpha}\sum_{t=1}^T \|G_t\|_{L_t^{-1}}^2 +\frac{\alpha}{2}\epsilon \mathcal{D}^2+\frac{\alpha}{2}\sum_{t=1}^T\|\W_t-\W^*\|_{G_tG_t^\top}^2.\\
\end{align*}
Thus we have
\begin{align*}
    \sum_{t=1}^T \psi_t(\W_t)-\psi_t(\W^*)
    \leq\sum_{t=1}^T \langle G_t,\W_t-\W^*\rangle-\frac{\alpha}{2}\|\W^*-\W_t\|_{G_tG_t^\top}^2
    \leq\frac{1}{2\alpha}\sum_{t=1}^T \|G_t\|_{L_t^{-1}}^2 +\frac{\alpha\epsilon \mathcal{D}^2}{2}.
\end{align*}
\end{proof}
We next present the matrix-form elliptical potential lemma. The proof is similar to the proof in \cite{hazan2007}. For completeness, we provide the proof here.

\begin{lemma}\label{lemma: elliptical potential}
Suppose $L_0=\epsilon I$, $L_{t+1}=L_t+G_tG_t^\top$ and $\|G_t\|_2\leq \mathcal{L}_G$ for all $t$, we have
\begin{align}\label{elliptical potential}
    \sum_{t=1}^T \|G_t\|_{L_t^{-1}}^2\leq \log \frac{|L_T|}{|L_0|}\leq n\log\left(\frac{T\mathcal{L}_G^2+\epsilon}{\epsilon}\right).
\end{align}
\end{lemma}

\begin{proof}
By the matrix inequality $\langle A-B,B\rangle\geq\log\frac{|A|}{|B|}$ for $A\succeq B\succeq0$, we have
\begin{align*}
    \sum_{t=1}^T \|G_t\|_{L_t^{-1}}^2 &= \sum_{t=1}^T \langle G_tG_t^\top, L_t\rangle\leq \sum_{t=1}^T \langle L_t-L_{t-1}, L_t\rangle\leq \sum_{t=1}^T \log\frac{|L_t|}{|L_{t-1}|}= \log \frac{|L_T|}{|L_0|}.
\end{align*}

Since $L_T=\sum_{t=1}^T G_tG_t^\top+\epsilon I$ and $\|G_t\|_2\leq \mathcal{L}_G$, the largest eigenvalue of $L_T$ is at most $T\mathcal{L}_G^2+\epsilon$. This finishes the proof.
\end{proof}
Finally, we can prove the logarithmic regret of our proposed method by combining all the lemmas together and setting an appropriate regularization coefficient.

\begin{proof}
By Lemma \ref{lemma: regret bound} and Lemma \ref{lemma: elliptical potential}, we have
\begin{align*}
    \mathcal{R}_T&\leq\frac{1}{2\alpha}\sum_{t=1}^T \|G_t\|_{L_t^{-1}}^2 +\frac{\alpha\epsilon\mathcal{D}^2}{2}\leq \frac{1}{2\alpha}n\log(\frac{T\mathcal{L}_G^2+\epsilon}{\epsilon})+\frac{\alpha\epsilon \mathcal{D}^2}{2}
    \leq \frac{n}{\alpha}\log \alpha \mathcal{L}_G^2\mathcal{D}^2T,
\end{align*}
which completes our proof.
\end{proof}

%%%%%%%%%%%%%%%%%%%%%%%%%%%%%%%%%%%%%%%%%%%%%%%%%%%%%%%%%%%%

\end{document}